\documentclass[letterpaper,12pt]{amsart}
\usepackage{amsmath,amssymb,amsxtra,amsthm, amstext, amscd,amsfonts,fancyhdr, hyperref,bbold}
\newcommand{\bA}{{\mathbb{A}}}

\newcommand{\bF}{{\mathbb{F}}}

\newcommand{\bP}{{\mathbb{P}}}

\newcommand{\bR}{{\mathbb{R}}}

\newcommand{\bZ}{{\mathbb{Z}}}

\newcommand{\Ba}{{\mathbf{a}}}
\newcommand{\Bb}{{\mathbf{b}}}

\newcommand{\Bh}{{\mathbf{h}}}

\newcommand{\Bt}{{\mathbf{t}}}

\newcommand{\Bx}{{\mathbf{x}}}
\newcommand{\By}{{\mathbf{y}}}
\newcommand{\Bz}{{\mathbf{z}}}

  \newcommand{\D}{{\mathcal{D}}}

  \newcommand{\G}{{\mathcal{G}}}

  \newcommand{\R}{{\mathcal{R}}}
\renewcommand{\S}{{\mathcal{S}}}

\newcommand{\rank}{\operatorname{rank}}

\newcommand{\sing}{\operatorname{sing}}

\newcommand{\upchi}{{\raise.35ex\hbox{$\chi$}}}

\newcommand{\Mod}[1]{\ (\mathrm{mod}\ #1)}

\makeatletter
\@namedef{subjclassname@2010}{%
  \textup{2010} Mathematics Subject Classification}
\makeatother

\newtheorem{theorem}{Theorem}[section]
\newtheorem{corollary}[theorem]{Corollary}

\newtheorem{lemma}[theorem]{Lemma}

\theoremstyle{definition}

\numberwithin{equation}{section}

\begin{document}

\title[An exponential sum modulo $p^2$]{A stratification result\\ for an exponential sum modulo $p^2$}

\author{Kostadinka Lapkova}
\address{Institute of Analysis and Number Theory\\
Graz University of Technology \\
Kopernikusgasse 24/II\\
8010 Graz\\
Austria}
\email{lapkova@math.tugraz.at}

\author{Stanley Yao Xiao}
\address{Department of Mathematics \\
University of Toronto \\
Bahen Centre \\
40 St. George St., Room 6290 \\
Toronto, Ontario, Canada M5S 2E4}
\email{stanley.xiao@math.toronto.edu}
\indent



\begin{abstract} In this note we consider algebraic exponential sums over the values of homogeneous nonsingular polynomials $F(x_1, \cdots, x_n) \in \bZ[x_1, \cdots, x_n]$ in the quotient ring $\bZ/p^2\bZ$. We provide an estimate of this exponential sum and a corresponding stratification of the space $\bA_{\bF_p}^n$, which in particular illustrates a general stratification theorem of Fouvry and Katz. 
\end{abstract}

\date\today

\maketitle


\section{Introduction}
Let $F(\Bx) \in \bZ[x_1, \cdots, x_n]$ be a polynomial of degree $d \geq 1$ and let $V_F$ be the hypersurface defined by $F(\Bx)=0$. Denote by $K$ a finite field or a quotient ring of size $q$. We are interested in the exponential sums 
\begin{equation} \label{main exp} \S_F(\Bh;K)=\sum_{\Bx \in V_F(K)} \exp \left(\frac{h_1 x_1 + \cdots + h_n x_n}{q}\right)\,,\end{equation}
where we use the notation $\exp(x):=e^{2\pi i x}$. 
In this note we address the estimate of $\S_F(\Bh,\bZ/p^2 \bZ)$ for large primes $p$.\\ 

A fundamental work in the estimation of general algebraic exponential sums is the paper of Fouvry and Katz \cite{FK}. When we drop a multiplicative character factor from the sums considered in \cite{FK}, and take not a general nontrivial additive character but only $\exp\pmod q$, we arrive at sums similar to $\S_F(\Bh;\bF_p)$, thus our result would be a special case of the sums considered in \cite{FK}, but with the modulus being squares of primes rather than primes. An analogous algebraic exponential sum was treated by Fouvry and our main theorem resembles \cite[Proposition 1.0]{F00} in the special case when the subscheme considered is a smooth hypersurface. \\

One significant aspect of the results of Fouvry and Fouvry-Katz \cite{F00, FK} is that they establish the existence of a decreasing filtration of subschemes (stratification) $\bA_{\bZ}^n\supset \G_1\supset\G_2\supset\ldots\supset\G_n$ of codimension $\geq j$, such that whenever $\Bh\notin\G_j(\bF_p)$ we have 
\begin{equation}\label{good}\S_F(\Bh,\bF_p)\ll p^{(n-1)/2+(j-1)/2}.
\end{equation} 
Thus the deviation from the bound $p^{(n-1)/2}$, which amounts to square-root cancellation and is the best that can be expected in general, is tamed by the fact that the dimensions of the subschemes $\G_j$ for $j\geq 1$ are small. However, the construction of this stratification is ineffective and inaccessible for non-specialists. Furthermore, the verifying of certain criteria that allow to apply the theorems of Fouvry-Katz \cite{FK} constitutes a separate problem, see for example \cite{devin}. Also see  \cite{xu} for a certain generalization of \cite{FK}. A further overview of the topic is given in the excellent review by P.~Michel \cite{michelMScN}.\\

On the other hand, Stepanov \cite{stepanov} focuses not on the stratification but rather constructs a subset of the good locus $\bar\G_1$ in which the optimal bound $\S_F(\Bh,\bF_q)\ll q^{(n-1)/2+\epsilon}$ holds for any finite field and not only for fields of prime order. Heath-Brown \cite{Heath-Brown85} treats the special case $q=p^2$ and $F\equiv \mathbb{1}$.\\

We show that in the specific case considered in (\ref{main exp}) the expected estimate holds and we give an explicitly constructed stratification.\\



For a given vector $\Bh \in \bZ^n$ consider the auxiliary affine variety $W_{F,\Bh}$ given by the equations defined over $\bF_p$
\begin{eqnarray}F(\By) &=& 0\nonumber\\ h_i \frac{\partial F}{\partial x_j} (\By) &=& h_j \frac{\partial F}{\partial x_i} (\By),\quad 1\leq i,j \leq n, i\neq j.
\end{eqnarray}
Note that this is the singular locus of the variety $\left\{F(\By)=0, \Bh\cdot\By=0 \right\}$. Now put $\G_{F,j}$ for the set 
given by
\begin{equation}\label{def G}\G_{F,j} = \{\Bh \in \bF_p^n : \dim(W_{F,\Bh}) \geq j\}.\end{equation}
We have the following theorem, which can be viewed as an extension to Theorem 1.2 of Fouvry-Katz \cite{FK} to $V(\bZ/p^2 \bZ)$ in a restricted setting. 

\begin{theorem} \label{thm} Let $F$ be a non-singular homogeneous polynomial in $\bZ[x_1, \cdots, x_n]$. Let $V_F$ be the hypersurface defined by $F$. Further, suppose that $F$ is non-singular modulo $p$. For $\Bh\in\bZ^n$ let $S(\Bh;p^2)$ be the exponential sum
	\begin{equation} \label{exp sum 0} S(\Bh; p^2) = \sum_{\Bx \in V_F(\bZ/p^2 \bZ)} \exp \left(\frac{h_1 x_1 + \cdots + h_n x_n}{p^2}\right). \end{equation}	
	
	Then the sets defined in \eqref{def G} are varieties satisfying $\G_{F,0}=\bA_{\bF_p}^n\supseteq \G_{F,1}\supseteq\ldots\supseteq\G_{F,n}$, where the codimension of each $\G_{F,j}$ is $\geq j$, and we have the estimate
              \[ S(\Bh; p^2)  = O_{F} \left(p^{n + j-2}\right)\]
	whenever $\Bh\notin \G_{F,j}(\bF_p)$.
\end{theorem}
The interest in this theorem is that unlike in the situation considered by Fouvry and Katz, a priori we have relatively little access to tools from algebraic geometry. This is because the base ring we consider is $\bZ/p^2 \bZ$ which is not a field, so many of the basic facts one takes for granted from algebraic geometry do not necessarily hold. Nevertheless, we are able to reduce to a case where we once again consider things over a field. \\


When $n = 2$ and $F$ is non-singular of degree $d\geq 2$, we have $\dim W_{F, \Bh} = 1$ if and only if the equation $h_2 \frac{\partial F}{\partial x_1}(\By) = h_1 \frac{\partial F}{\partial x_2} (\By)$ is redundant. If $h_1, h_2$ are not both zero then this implies that $F$ shares a component with a curve of strictly smaller degree, which violates the assumption that $F$ is non-singular. Therefore the only way to have $\dim W_{F,\Bh} = 1$ is for $h_1 = h_2 = 0$ and so any non-singular $F \in \bZ[x_1, x_2]$ we have square-root cancellation, and obtain the bound
\[S(\Bh;p^2)\ll p^{n-1}.\] 
We will soon see in the proof that this essentially follows from the Weil bound resolving the Riemann Hypothesis for curves, where for any (higher) dimension of $V_F$ we apply the theorem of Lang-Weil. \\

In many sieving problems (see for example \cite{Hoo1}) one applies a result like that of Fouvry-Katz \cite{FK} to use congruence conditions modulo somewhat large primes to show paucity of contributions, which results in sums of the shape
\[\sum_{x_1 < p < x_2} \left(\frac{B^n}{p} + 1 \right).\]
Notice that the sum over the lead terms $1/p$ diverges, so in such a case the size of $x_1$ is not relevant. On the other hand, if we apply our Theorem instead, then one would end up with a sum over $1/p^2$ which is then convergent, so the sum would then be sensitive to the size of $x_1$. This is crucial if one wishes to obtain asymptotic formulae. Similar arguments could lead to some of the potential applications of our theorem.

\section{Proof of Theorem \ref{thm}}
\label{proof}

Write $\Bx = \By + p \Bz$. Then we have
\begin{equation} F(\Bx) = F(\By) + p \left(\sum_{i=1}^n z_i \frac{\partial F}{\partial x_i} (\By) \right) + p^2 H_F(\By; \Bz)\equiv F(\By)+p\Bz\cdot\nabla F(\By)\pmod {p^2}.
\end{equation}
Now restrict $\By \in \bF_p$. Then $\By + p \Bz \in V\left(\bZ/p^2 \bZ\right)$ if and only if $y\in V(\bF_p)$ and 
\[\Bz\cdot\nabla F(\By) \equiv - F(\By)/p \pmod{p}.\]
Then (\ref{exp sum 0}) becomes 
\[\sum_{\By \in V(\bF_p)} \exp\left(\frac{h_1 y_1 + \cdots + h_n y_n}{p^2}\right) \sum_{\substack{\Bz\Mod p\\ \Bz\cdot\nabla F(\By)\equiv - F(\By)/p \Mod{p}}} \exp \left(\frac{h_1 z_1 + \cdots + h_n z_n}{p}\right). \]
We examine the inner sum in detail. Let us introduce the notation
\[T_p(\Ba,\Bb,c):=\sum_{\substack{\Bz\Mod p\\ \Bb\cdot\Bz\equiv c\Mod p}} \exp\left(\frac{\Ba\cdot\Bz}{p}\right).\]

We claim that the exponential sum $T_p(\Ba,\Bb,c)=0$ every time there exists a pair $(i,j)$ with $1\leq i,j\leq n$ and $i\neq j$, such that $a_ib_j\neq a_jb_i\Mod p$. Indeed, if the latter holds, we can assume that $b_i\neq 0$, since $b_i=b_j=0$ contradicts the assumption. We can assume that $i=n$. We write 
\[z_n\equiv -b_n^{-1}\left(b_1z_1+\ldots+b_{n-1}z_{n-1}-c\right)\Mod p\] 
and then $T_p(\Ba,\Bb,c)$ equals 
\begin{eqnarray*}\sum_{z_1,\ldots,z_{n-1}\Mod p} \exp\left(\frac{a_1z_1+\ldots a_{n-1}z_{n-1}-a_nb_n^{-1}\left(b_1z_1+\ldots+b_{n-1}z_{n-1}-c\right)}{p}\right)
	=\\
	\sum_{z_1,\ldots,z_{n-1}\Mod p} \exp\left(\frac{\sum_{1 \leq j \leq n-1}(a_j-a_nb_n^{-1}b_j)z_j+a_nb_n^{-1}c}{p}\right).
\end{eqnarray*}
As we assumed that for some $j\in\left[1,n-1\right]$ we have $a_j-a_nb_n^{-1}b_j\not\equiv 0\pmod p$, it follows that $T_p(\Ba,\Bb,c)=0$. On the other hand, if $a_ib_j\equiv a_jb_i\Mod p$ for every $1\leq i,j\leq n$ and $i\neq j$, we have trivially $|T_p(\Ba,\Bb,c)|= p^{n-1}$. \\

Now we conclude that \eqref{exp sum 0} is equal to
\[\sum_{\By \in V(\bF_p)} \exp\left(\frac{\Bh\cdot \By}{p^2}\right) T_p(\Bh,\nabla F(\By),-F(\By)/p)\]
and this is bounded from above by 
\[p^{n-1}\#\left\lbrace \By\in\bF_p^n : F(\By)=0, \rank \left(\begin{array}{c} \Bh\\ \nabla F(\By)\end{array}\right)<2\right\rbrace=p^{n-1}\sum_{\By \in W_{F, \Bh}(\bF_p)} 1.\]
%
If $\Bh \notin \G_{F,j}(\bF_p)$, then 
\[\dim (W_{F,\Bh}) \leq j-1.\]
By the Lang-Weil theorem, it follows that
\[\sum_{\By \in W_{F, \Bh}(\bF_p)} 1 \ll_{d,n} p^{j-1}\]
and the estimate of the exponential sum $S(\Bh,p^2)$ follows. It remains to deal with the dimensions of the sets $\G_{F,j}$.\\

The following lemma is reminiscent of Heath-Brown's Lemma 2 in \cite{Heath-Brown}, proved using the method of dimension counting via incidence geometry. In \cite{Heath-Brown} he constructs a stratification similar to ours, but for exponential sums modulo $p$. The lemma will help us estimate the dimensions of $\G_{F,j}$, which turn out to be affine varieties.
\begin{lemma}\label{dimG} Let $F\in K\left[x_1,\ldots,x_n\right]$ be a non-singular homogenous polynomial of degree $d$ in $n$ variables defined over a field $K$. For a given $\Bz\in K^n$ let \[V_{\Bz}=\sing (X\cap H_\Bz),\] where $X=\lbrace F(\Bx)=0\rbrace\subset \bP_K^{n-1}$ and $H_\Bz=\lbrace \Bz\cdot\Bx=0\rbrace\subset \bP_K^{n-1}$. Let 
	\[T_s=\lbrace z=[\Bz]\in\bP_K^{n-1}: \dim(V_\Bz)\geq s\rbrace,\quad s\geq -1. \]
	Then $T_s$ is a projective variety of dimension at most $n-2-s$.	 
\end{lemma} 
\begin{proof}
	Consider the variety 
	\[I=\left\lbrace(z,\Bx)\in\bP_K^{n-1}\times X : \rank \left(\begin{array}{c} \Bz\\ \nabla F(\Bx)\end{array}\right)<2\right\rbrace.\]
	The projection $\pi_2:I\rightarrow X$ is onto, since $[\nabla F(\Bx)]$ is in the fibre for every $\Bx\in X$. Since $X$ is non-singular we have $\nabla F(\Bx)\neq 0$ for every $\Bx\in X$, therefore we have the condition \[ \rank \left(\begin{array}{c} \Bz\\ \nabla F(\Bx)\end{array}\right)=1, \]
	i.e. each minor in the matrix equals $0$. It is easy to see that this implies that any other $z\in\bP^{n-1}$ in the same fibre is such that $\Bz$ is proportional to $\nabla F(\Bx)$, i.e. $[z]=[\nabla F(\Bx)]$. This means that the fibre consists of a single point. Hence $I$ is irreducible and by the theorem on the dimension of fibres we get
	\[\dim(I)=\dim(\pi_2^{-1}(\Bx))+\dim(X)=0+n-2=n-2.\]
	Let $U$ be an irreducible component of $T_s$ and consider the projection 
	\[\pi_1:I\rightarrow U.\]
	Then 
	\[n-2=\dim(I)\geq s+\dim(U),\]
	which yields $\dim(T_s)\leq n-2-s$.

\end{proof}

We will also need the following estimate of the dimensions of $\G_{F,j}$.
\begin{lemma}
	For any homogenous and non-singular polynomial $F$ and $j\geq 0$ we have
	\[\dim(\G_{F,j}(\bA_{\bF_p}^n))\leq n-j\,.\]
\end{lemma}
\begin{proof}
	When we regard $X$ and $H_\Bz$ over $\bA_{\bF_p}^n$ as affine varieties, and respectively the affine singular locus of their intersection $\sing(X\cap H_\Bz)(\bA_{\bF_p}^n)=V_{\Bz}(\bA_{\bF_p}^n)$, we have $\dim(V_{\Bz}(\bA_{\bF_p}^n))=\dim(V_{\Bz})+1$. 
	Thus if $[\Bz]\in T_s$ for $s\geq -1$, we have $\dim(V_\Bz)\geq s$ and respectively $\dim(W_{F,\Bz}(\bA_{\bF_p}^n))\geq s+1$, so any of the corresponding $\Bz\in\bF_p^n$ satisfies $\Bz\in \G_{F,s+1}$. As $\dim(T_s)\leq n-2-s$ and $T_s$ is exactly $\G_{F,s+1}(\bP_{\bF_p}^{n-1})$, we have $\dim(\G_{F,s+1})\leq (n-2-s)+1=n-(s+1)$. We conclude that for $j\geq 0$ we have $\dim(\G_{F,j}(\bA_{\bF_p}^n))\leq n-j$.
	
	
\end{proof}
This completes the proof of Theorem \ref{thm}.


\subsection*{Acknowledgements} We are grateful to Prof. Tim Browning for helping with Lemma \ref{dimG}. K. Lapkova is supported by a Hertha Firnberg grant [T846-N35] of the  Austrian Science Fund (FWF).


\end{document}